\documentclass[a4paper, 11pt]{article}
\usepackage{amssymb, latexsym,ulem,amsrefs, amsmath, amsthm, color, marginnote,hyperref,cancel}
\usepackage{a4wide}

\newtheorem{theorem}{Theorem}[section]
\newtheorem{lemma}[theorem]{Lemma}
\newtheorem{proposition}[theorem]{Proposition}
\newtheorem{corollary}[theorem]{Corollary}
\newtheorem{question}[theorem]{Question}

\theoremstyle{definition}

\newtheorem{ex}[theorem]{Example}

\newcommand{\Z}{\mathbb{Z}}
\newcommand{\N}{\mathbb{N}}

\begin{document}
\title{Property $(\diamond)$ for Ore extensions of small Krull dimension}
\author{Ken Brown, Paula A.A.B. Carvalho and Jerzy Matczuk}
\date{\today}
\maketitle

\abstract{This paper is a continuation of a project to determine which skew polynomial algebras $S = R[\theta; \alpha]$ satisfy property $(\diamond)$, namely that the injective hull of every simple $S$-module is locally artinian, where $k$ is a field, $R$ is a commutative noetherian $k$-algebra, and $\alpha$ is a $k$-algebra automorphism of $R$. Earlier work (which we review) and further analysis done here leads us to focus on the case where $S$ is a primitive domain and $R$ has Krull dimension 1 and contains an uncountable field. Then we show first that if $|\mathrm{Spec}(R)|$ is infinite then $S$ does not satisfy $(\diamond)$. Secondly we show that when $R = k[X]_{<X>}$ and $\alpha (X) = qX$ where $q \in k \setminus \{0\}$ is not a root of unity then $S$ does not satisfy $(\diamond)$. This is in complete contrast to our earlier result that, when $R = k[[X]]$ and $\alpha$ is an arbitrary $k$-algebra automorphism of infinite order, $S$ satisfies $(\diamond)$. A number of open questions are stated.}



\section{Introduction}\label{intro}  A noetherian ring $A$ is said to {\it satisfy}  $(\diamond)$ if, for every simple $A$-module $V$, every finitely generated submodule of the injective hull of $V$ has finite composition length. The fact that commutative noetherian rings satisfy $(\diamond)$ is a consequence of the Artin-Rees lemma, and is part of the seminal work of Matlis  \cite{M}. The terminology was introduced in \cite{PaulaChristianDilek}; summaries of the considerable volume of research on the wider validity of $(\diamond)$  can be found in \cite{Ian} and in \cite{BCM}. The present paper continues our earlier work \cite{BCM}, \cite{BCM2} on the question of when $(\diamond)$ holds for skew polynomial algebras. Specifically, we fix a field $k$, a commutative noetherian $k$-algebra $R$ and a $k$-algebra automorphism $\alpha$ of $R$, and define $S$ to be the algebra $S := R[\theta ; \alpha]$, so $\theta$ is an indeterminate and $\theta r = \alpha(r)\theta$ for $r \in R$. The question of when $S$ satisfies $(\diamond)$ turns out to be surprisingly delicate, and is still far from completely answered.

The state of play for $(\diamond)$ following \cite{BCM}, \cite{BCM2} for the algebras $S$ is reviewed in $\S$\ref{survey}, the headline version going as follows. Let $\mathrm{Kdim}(R)$ denote the Krull dimension of $R$. To a large extent, the question of when $S$ satisfies $(\diamond)$ reduces to the case where $S$ is a primitive domain, (so $\alpha$ has infinite order), $V$ is a faithful simple $S$-module, $M$ is a finitely generated submodule of the injective hull of $V$, and one has to answer the question: does $M$ always have finite composition length? The answer is ``Yes'' if $R$ is a field and (provided $R$ contains an uncountable field) ``No'' if $\mathrm{Kdim}(R) \geq 2$ or $R$ has uncountably many prime ideals. This leaves the issue of what happens when $\mathrm{Kdim}(R)=1$ and $R$ has only countably many maximal ideals. For one such coefficient algebra, namely $k[[X]]$, we showed that the answer is ``Yes'' in \cite[Theorem 1.2]{BCM2}, restated below as Theorem \ref{power}.  

There are two main new results in the present paper, both of them addressing the gap in our knowledge when $\mathrm{Kdim}(R)=1$. The first, the novel part of Theorem \ref{better}, reduces the gap to the case where $R$ is semilocal:

\begin{theorem}\label{firstthm} Suppose that $S= R[\theta; \alpha]$ is a primitive domain, with $R$ containing an uncountable field. If $R$ has infinitely many prime ideals, then $S$ does not satisfy $(\diamond)$.
\end{theorem}

Our second new result considers the case where the domain $R$ with $\mathrm{Kdim}(R)=1$ is not just semilocal but {\it local}. As explained above we already know that $S$ satisfies $(\diamond)$ when $R = k[[X]]$. But, in contrast to this, we have as Theorem \ref{k[X]thm}:

\begin{theorem}\label{secondthm} Let $k$ be a field of characteristic not 2. Let $R = k[X]_{<X>}$, the localisation of $k[X]$ at its maximal ideal $<X>$. Let the automorphism $\alpha$ be given by $\alpha(X) = qX$,  where $q \in k \setminus \{0\}$ is not a root of unity. Then $S$ does not satisfy  $(\diamond)$. 
\end{theorem}

The proof of Theorem \ref{firstthm} is quite short, given our earlier results; it is given in $\S$\ref{reduction}. Much more work is needed to prove Theorem \ref{secondthm} in $\S$\ref{fail}. The method is similar to that used in \cite{BCM2} to prove $(\diamond)$ for $S = k[[X]][\theta; \alpha]$, despite the outcome being the opposite in Theorem \ref{secondthm}. Namely, the starting point remains detailed knowledge of the simple and the cyclic 1-critical $S$-modules, based on the work \cite{BVO} of Bavula and Van Oystaeyen. This knowledge allows us to translate the existence or otherwise of non-split exact sequences of such modules into a condition on the {\it monoid commutativity} of irreducible elements of $S$. We are then able to show that this condition fails for a specific choice of irreducible elements. 

A number of open questions are stated and discussed in $\S\S$\ref{survey} and \ref{reduction}.

\bigskip

\section{$(\diamond)$ for Ore extensions - state of play}\label{survey}
\subsection{Notation and setup}\label{notation} Throughout, unless other hypotheses are explicitly stated, $k$, $R$, $\alpha$, $S$ and $\theta$ will always be as in the first paragraph of the introduction, with $\mathrm{Kdim}(R)$ denoting the Krull dimension of $R$. We will usually work with {\it right} modules, but will suppress this adjective most of the time. If $A$ is a ring and $M$ is an $A$-module, the injective hull of $M$ as an $A$-module is denoted by $E_A(M)$. Extending the terminology of the Introduction, if $M$ is a simple $A$-module and every finitely generated submodule of $E_A(M)$ has finite composition length, we say that $M$ {\it satisfies $(\diamond)$ as an $A$-module}, suppressing the last three words when they are clear from the context. Thus the algebra $S$ satisfies ({\it right}) $(\diamond)$ if every simple (right) $S$-module satisfies $(\diamond)$ as an $S$-module. The motivating issue of this paper and its two predecessors \cite{BCM}, \cite{BCM2} is:

\begin{question}\label{keyqn} Let $k$, $R$, $\alpha$ and $S$ be as above. When does $S$ satisfy $(\diamond)$?
\end{question}

If $I$ is an ideal of a ring $T$ and $V$ is a $T/I$-module then clearly $E_{T/I}(V) \subseteq E_T(V)$. Hence, given a noetherian ring $A$ and a simple $A$-module $V$ with annihilator ideal $P$, 
\begin{equation}\label{reduce} V \textit{ satisfies } (\diamond) \textit{ as an } A\textit{-module only if } V \textit{ satisfies } (\diamond) \textit{ as an } A/P \textit{-module}.
\end{equation}
\noindent  One might hope that the converse to (\ref{reduce}) is also true, but this is not the case: in \cite{GS} Goodearl and Schofield constructed a noetherian algebra $T$ with $\mathrm{Kdim}(T) = 1$ having a simple module $V$ with $ P := \mathrm{Ann}(V)$ a minimal prime of $T$, such that $E_{T/P}(V)$ is locally artinian, but $E_T(V)$ is not locally artinian. Nevertheless, in large measure (though, as we shall explain in $\S$\ref{general}, not completely), Question \ref{keyqn} reduces to the case where the Ore extension $S$ is a primitive domain. We are thus led to consider the primitive ideals $P$ of $S = R[\theta; \alpha]$. We summarise the relevant facts about these ideals in $\S$\ref{previous}.

\medskip

\subsection{$(\diamond)$  when $S$ is a primitive domain}\label{previous}  Each primitive ideal $P$ of $S  = R[\theta; \alpha]$ is either {\it not induced} or {\it induced}: that is,
$$ \textit{either }{\bf (I)}\;(P \cap R)S \subsetneq P; \;  \textit{ or } \; {\bf (II)} \;   (P \cap R)S = P.$$

\noindent For a primitive ideal $P$ of $S$ of Type ${\bf (I)}$ it follows from a result of Irving \cite{I} that $\alpha$ induces an automorphism of finite order on $R/P \cap R$. This implies that $S/(P \cap R)S$ satisfies a polynomial identity \cite[Corollary 10]{DS}, so that $S/P$ is artinian by Kaplansky's theorem \cite[Theorem I.13.4]{BG}. From this it can be shown using arguments related to the second layer condition that the converse to (\ref{reduce}) holds in this situation, so that $E_S(V)$ is locally artinian where $V$ denotes the unique simple $S/P$-module; see \cite[Theorem 3.4]{BCM} for details. 

\medskip

Consider on the other hand a primitive ideal $P$ of $S$ is of Type {\bf (II)}. First we recall that such induced primitive ideals of $S$ are completely characterised by a result of Leroy and Matczuk, (which is valid for an arbitrary commutative noetherian ring $R$):

\begin{theorem}\label{LMresult}\cite[Theorem 4.2]{LM} The ring $ S = R[\theta; \alpha]$ is primitive if and only if $|\alpha| = \infty$ and there exists an element $a \in R$ such that, given any non-zero $\alpha$-invariant ideal $I$ of $R$, there exists $n \geq 1$ such that $0 \neq a\alpha(a) \cdots \alpha^{n-1}(a) \in I$. 
\end{theorem}

\noindent An element $a$ as in the theorem  is called an $\alpha$-{\it special element of} $R$, and if $R$ has an $\alpha$-special element then it is called an $\alpha$-{\it special ring}. Clearly if $R$ is $\alpha$-{\it simple} (meaning that it has no proper $\alpha$-invariant ideal), then it is $\alpha$-special. We'll encounter examples of $\alpha$-special rings which are not $\alpha$-simple in $\S$\ref{powerseries} and $\S$\ref{fail}. 

In analysing the validity of $(\diamond)$ for $S/P \cong (R/P \cap R)[\theta ; \alpha]$ we may  assume that $R/(P \cap R)$, and hence $S/P$, are domains, as follows. Recall that an $\alpha$-invariant ideal $I$ of $R$ is called $\alpha$-{\it prime} if, given $\alpha$-invariant ideals $A$ and $B$ of $R$, $AB \subseteq I$ only if $A \subseteq I$ or $B \subseteq I$. If $R$ is $\alpha$-prime then it is semiprime and $\alpha$ permutes its minimal primes $\mathfrak{p}_1, \ldots , \mathfrak{p}_t$ transitively. Hence  $\alpha^t (\mathfrak{p}_i) = \mathfrak{p}_i$ for all $i$ and $\alpha^t$ induces an automorphism of $R/\mathfrak{p}_i$. Clearly if $P$ is a prime ideal of $S$ then $P \cap R $ is an $\alpha$-prime ideal of $R$. It is now easy to prove
\begin{lemma}\label{domain}(\cite[Lemma 4.5]{BCM}) Suppose that $R$ is $\alpha$-prime with minimal primes $\mathfrak{p}_1, \ldots , \mathfrak{p}_t$. Then $S = R[\theta; \alpha]$ satisfies $(\diamond)$ if and only if $\widehat{S} := (R/\mathfrak{p}_1)[\theta^t ; \alpha^t]$ satisfies $(\diamond)$. Moreover, if $S$ is primitive then so is $\widehat{S}$.
\end{lemma} 

\medskip

Here in simplified form is our key previous result concerning $(\diamond)$ for a primitive skew polynomial domain.

\begin{theorem}\label{primitive}\cite[Theorem 5.4]{BCM} Suppose that $S = R[\theta; \alpha]$ is a primitive domain.
\begin{enumerate}
\item[(i)] If $R$ is a field then $S$ satisfies $(\diamond)$.
\item[(ii)] Suppose that $R$ contains an uncountable field. Suppose also that either $R$ has Krull
dimension at least 2, or $\mathrm{Spec}(R)$ is uncountable. Then $S$ does not satisfy $(\diamond)$.
\end{enumerate}
\end{theorem}

\noindent For primitive skew polynomial domains there is thus a glaring gap between the sufficient condition for $(\diamond)$ given by $(i)$ of the theorem and the necessary conditions provided by $(ii)$. Leaving aside the issue of the presence or otherwise of an uncountable field we therefore formulate the following crucial special case of Question \ref{keyqn}.

\begin{question}\label{primitiveqn} Let $R$ be a commutative noetherian domain of Krull dimension 1 with $\mathrm{Spec}(R)$ countable, and let $\alpha \in \mathrm{Aut}(R)$ with $|\alpha| = \infty$ be such that $R$ is $\alpha$-special, so that $S := R[\theta ; \alpha]$ is primitive. Under what circumstances does $S$ satisfy $(\diamond)$?
\end{question}

\noindent If $R$ is a commutative noetherian domain which is not a field and $\alpha \in \mathrm{Aut}(R)$ with $|\alpha| = \infty$ is such that $R$ is $\alpha$-simple, then $S = R[\theta;\alpha]$ does not satisfy $(\diamond)$, by \cite[Proposition 4.10]{BCM}. So when addressing Question \ref{primitiveqn} we may assume that $R$ is {\it not} $\alpha$-simple. 

We shall address aspects of Question \ref{primitiveqn} in $\S\S$\ref{powerseries}, \ref{reduction} and \ref{fail}. But before doing so we briefly review in $\S$\ref{general} the matter of passing from the primitive case of $(\diamond)$ to the general case.

\medskip

\subsection{From the primitive case to $(\diamond)$}\label{general}

We do not know if the converse to (\ref{reduce}) is true when $S$ is a skew polynomial algebra as in $\S$\ref{survey}; that is, the following is open:

\begin{question}\label{reverse} Let $S = R[\theta; \alpha]$, and let $V$ be a simple $S$-module with $\mathrm{Ann}_S(V) = P$. If $E_{S/P}(V)$ is locally artinian, is $E_S(V)$ locally artinian?
\end{question}

As is more or less apparent from the previous subsection, the difficulty with Question \ref{reverse} arises when $P$ is type {\bf II}, that is $P = (P \cap R)S$. It is not hard to see that in this case $P$ satisfies the Artin-Rees property, from which it follows that 
$$ E_S(V) \; = \; \bigcup_{n \geq 1} E_{S/P^n}(V). $$
Thus to give a positive answer to Question \ref{reverse} one must show that $E_{S/P^n}(V)$ is locally artinian for all $n$, which leads one via \cite[Lemma 1]{RZ} to issues about the bimodule structure of $P^{n-1}/P^n$. In many important cases, for example when $S$ is polynormal \cite[Corollary 3.5]{BCM}, these matters can be handled to yield a positive answer to Question \ref{reverse}. Moreover the problem can also be circumvented when $R$ is an affine $k$-algebra:

\begin{theorem}\label{affine}\cite[Theorem 6.1, Remark $(ii)$ of $\S$6]{BCM} Let $R$ be a commutative affine $k$-algebra, $S = R[\theta; \alpha]$ where $\alpha$ is as usual a $k$-algebra automorphism. Suppose that $k$ is either uncountable or has characteristic 0.Then $S$ satisfies $(\diamond)$ if and only if $\mathrm{dim}_k(V) < \infty$ for every simple $S$-module (if and only if $S$ has no induced primitive ideals). 
\end{theorem}

It transpires that the question of when every simple $S$-module is finite dimensional is itself very delicate; see \cite[$\S$6, $\S$7]{BCM} for a discussion and an analysis of some examples.

\medskip

\subsection{The case $R = k[[X]]$}\label{powerseries}

The first examples living in the gap between $(i)$ and $(ii)$ of Theorem \ref{primitive} whose membership of $(\diamond)$ was determined were the following:

\begin{theorem}\label{power}(\cite[Theorem 1.4]{BCM2}) Let $R = k[[X]]$, where $k$ is any field, and let $\alpha$ be a $k$-algebra automorphism of $R$. Let $S = R[\theta; \alpha]$. Then $S$ satisfies $(\diamond)$.
\end{theorem}

Notice that, for any choice of $\alpha$, $X$ is an $\alpha$-special element; but $R$ is never $\alpha$-simple. The proof of Theorem \ref{power} in the non-trivial case when $|\alpha| = \infty$ hinges on a classification of the cyclic 1-critical and the simple $S$-modules, building on work of Bavula and Van Oystaeyen \cite{BVO}, combined with an analysis of the existence of non-split extensions of such modules using a type of ``monoidal commutativity'' criterion for the irreducible elements of $S$. Since we will need these ideas in proving Theorem \ref{k[X]thm} below, we recall the basic facts at the start of $\S$\ref{fail}. 

\bigskip

\section{Reduction to $R$ semilocal when $\mathrm{Kdim}(R) = 1$}\label{reduction}

In this section we improve Theorem \ref{primitive} by reducing the gap between its necessary and sufficient conditions for $(\diamond)$. Namely, we shall prove

\begin{theorem}\label{better}Suppose that $S = R[\theta; \alpha]$ is a primitive domain.
\begin{enumerate}
\item[(i)] If $R$ is a field then $S$ satisfies $(\diamond)$.
\item[(ii)] Suppose that $R$ contains an uncountable field. Suppose also that either $R$ has Krull
dimension at least 2, or $\mathrm{Spec}(R)$ is infinite. Then $S$ does not satisfy $(\diamond)$.
\end{enumerate}
\end{theorem}

\medskip

\noindent This improvement follows immediately from the combination of Theorem \ref{primitive} with

\begin{proposition}\label{Jelisi}
Let $R$ be a commutative noetherian domain with $\mathrm{Kdim}(R) = 1$ and let $\alpha$ be an automorphism of $R$ such that $R$ is $\alpha$-special.  
\begin{enumerate}
 \item[(i)]  $\mathrm{Spec}(R)$ has only finitely many finite $\alpha$-orbits. In particular, if every $\alpha$-orbit in $\mathrm{Spec}(R)$ is finite then $\mathrm{Spec}(R)$ is finite.
 \item[(ii)]  If $|\mathrm{Spec}(R)| = \infty$ then $S$ does not satisfy $(\diamond)$.
  
\end{enumerate}

\end{proposition}

\begin{proof}$(i)$ If $|\alpha| < \infty$ then it is easy to see that $\mathrm{Spec}(R)$ is finite, so we may assume that $|\alpha| = \infty$ has infinite order.  Let $\mathcal{O}_1$ be a finite $\alpha$-orbit in $\mathrm{Spec}(R)\setminus \{0\}$  and set $$I_{\mathcal{O}_1} \; := \; \bigcap\{\mathfrak{m} : \mathfrak{m} \in \mathcal{O}_1\}.$$
Let $a $ be an $\alpha$-special element. Since $I_{\mathcal{O}_1}$ is non-zero and $\alpha$-stable, there exists $n \geq 1$ with
$$ a\alpha(a) \cdots \alpha^{n-1}(a) \in I_{\mathcal{O}_1}.$$
Let $\mathfrak{m} \in \mathcal{O}_1$. Then there exists $j \in \mathbb{Z}$ such that $a \in \alpha^j(\mathfrak{m}) := \mathfrak{m}_1$. Similarly, each finite $\alpha$-orbit $\mathcal{O}_i$ in $\mathrm{Spec}(R)\setminus \{0\}$ contains a maximal ideal $\mathfrak{m}_i$ with $a \in \mathfrak{m}_i$. Since distinct $\alpha$-orbits are disjoint $\mathfrak{m}_i = \mathfrak{m}_j$ only if $i = j$. But since $R$ is a domain of Krull dimension 1 only finitely many maximal ideals can contain $a$. There are therefore only finitely many finite $\alpha$-orbits in $\mathrm{Spec}(R)$.

\medskip

\noindent $(ii)$ Suppose that $|\mathrm{Spec}(R)| = \infty$ and, in view of $(i)$, let $\mathcal{P}$ be the union of the finitely many finite $\alpha$-orbits in $\mathrm{Spec}(R)\setminus \{0\}$. Then $\mathcal{P}$ is either finite or empty, so we can define the non-zero ideal
$$ I \; := \; \bigcap_{P \in \mathcal{P}}P.$$

\noindent Suppose that $\mathcal{P}=\emptyset$. Then  $I=R$ and so $R$ is $\alpha$-simple. By \cite[Proposition 4.10]{BCM} $S$ does not satisfy $(\diamond)$. Assume now that $\mathcal{P}$ is non-empty. By hypothesis there is an infinite orbit of a maximal ideal $\mathfrak{m}$. Suppose that
\begin{equation}\label{failure}\bigcap_{P \in \mathcal{P}}P\; \subseteq \; \bigcup_{i\in\Z} \alpha^i(\mathfrak{m}).
\end{equation} 
By \cite[Proposition 2.5]{SharpVamos} every commutative noetherian ring containing an uncountable field has the countable prime avoidance property, so $\bigcap_{P \in \mathcal{P}}P\subseteq \alpha^j(\mathfrak{m})$ for some $j$. Thus $Q=\alpha^j(\mathfrak{m})$ for some $Q \in \mathcal{P}$, contradicting the fact that $Q$ has a finite orbit and $\mathfrak{m}$ an infinite one. Therefore (\ref{failure}) is false, so we can choose $y\in \bigcap_{P \in \mathcal{P}}P$ such that $y\notin \bigcup_{i\in\Z}\alpha^i(\mathfrak{m})$. Note that $y$ is an $\alpha$-special element and if $\mathcal A$ is the multiplicative set generated by $\alpha^i(y)$ for $i\in\Z$, $R{\mathcal A}^{-1}$ is not a field since $\mathfrak{m}\cap {\mathcal A}=\emptyset$. The result follows from \cite[Proposition 5.3]{BCM}.

\end{proof}

\medskip

As an aside we note that the following example shows $\alpha$-speciality of $R$ was essential in Proposition \ref{Jelisi}.

\begin{ex}

Let $p$ be a prime and consider $\mathbb{F}_p$ the field with $p$ elements. Let $\alpha$ be the Frobenius automorphism on the algebraic closure of $\mathbb{F}_p$, which we denote by $\overline{\mathbb{F}_p}$, so $\alpha(a)=a^p$ for any $a\in \mathbb{F}_p$. Extend $\alpha$ to $\overline{\mathbb{F}_p}[x]$ by setting $\alpha(x)=x$. Note that $\alpha$ is of infinite order but for each $w\in {\overline{\mathbb{F}_p}}$, there is $n_w\in \N$ such that $\alpha^{n_w}(w)=w$.  Since every maximal ideal is of the form $<x-w>$ for some $w\in \overline{\mathbb{F}_p}$, each maximal ideal has finite $\alpha$-orbit and there are an infinite number of maximal ideals.

But ${\overline{\mathbb{F}_p}}[x]$ is not $\alpha$-special: this follows from the proposition, but can easily be seen directly. For, assume otherwise and let $p(x)\in {\overline{\mathbb{F}_p}[x]}$ be an $\alpha$-special element. Then for each $Q\in \mathrm{Maxspec}(R)$, there exists $n_Q$ such that $p(x)\in \alpha^{n_Q}(Q)$. Since ${\overline{\mathbb{F}_p}[x]}$ is a PID, there are infinitely many maximal ideals and each has a finite orbit, $p(x)$ is divided by an infinite number of elements of ${\overline{\mathbb{F}_p}[x]}$, a contradiction.
\end{ex}

\medskip

To sum up: our analysis up to this point leads to the following further refinement of Question \ref{primitiveqn}:

\begin{question}\label{primitiveqn2} Let $R$ be a commutative noetherian semilocal domain of Krull dimension 1 containing an uncountable field, and let $\alpha \in \mathrm{Aut}(R)$ with $|\alpha| = \infty$,  (so $R$ is automatically $\alpha$-special and hence $S := R[\theta ; \alpha]$ is primitive). Does $S$ satisfy $(\diamond)$?
\end{question}

\medskip

In view of Theorem \ref{power} one might guess - as we did - that the answer to Question \ref{primitiveqn2} is ``yes''. But in fact this is wrong, as we show via the theorem of the next section.

\bigskip

\section{$(\diamond)$ fails for $R = k[X]_{\langle X \rangle}$}\label{fail}

Throughout this section $R$ will denote the local ring $k[X]_{<X>}$, $q \in k \setminus \{0\}$ and $\alpha$ will be the $k$-algebra automorphism of $R$ given by $\alpha (X) = qX$, with $S = R[\theta; \alpha]$. The aim is to prove 

\begin{theorem}\label{k[X]thm} With the notation as above, suppose that $q$ is not a root of unity and the characteristic of $k$ is not 2. Then $S$ does not satisfy  $(\diamond)$. 
\end{theorem}

\medskip

Our strategy is the same as that used in \cite{BCM2} to prove Theorem \ref{power} - but of course with the opposite conclusion to the earlier result. We'll explain this in more detail, with the necessary definitions, in the following paragraphs. But first, a very brief sketch: we recall from \cite{BCM2} that the classification (up to a suitable equivalence relation) of the simple and the cyclic 1-critical $S$-modules provided by the work of Bavula and Van Oystaeyen \cite{BVO} can be used to show that the existence of non-split extensions of simple $S$-modules by 1-critical $S$-modules is equivalent to the failure of a type of Ore-like condition on the set of irreducible elements of $S$, a condition which we call {\it monoid commutativity}. Then we prove Theorem \ref{k[X]thm} by demonstrating this failure for an explicit pair of elements. 

\medskip

It is important to note that all the results of \cite[$\S\S$5, 6.1, 6.2]{BCM2} are stated and proved for $S := R[\theta;\alpha]$ with $R$ a local noetherian discrete valuation ring with maximal ideal $XR$ and with $R/XR \cong k$; so in particular they are valid for both $S= k[[X]][\theta; \alpha]$, as applied in \cite[$\S$6.3]{BCM2} and for $S = k[X]_{<X>}[\theta; \alpha]$ as here.

\medskip

An {\it irreducible element} $w$ of a ring $W$ is a non-zero non-unit of $W$ such that, whenever $w = ab$ with $a,b \in W$, then either $a$ or $b$ is a unit of $W$. We recall from \cite[$\S$4.2]{BCM2} the classification of the irreducible elements used for $k[[X]][\theta; \alpha]$; exactly the same classification can be used for our current algebra $S := k[X]_{<X>}[\theta; \alpha]$. Namely, since $ S/X S \cong k[\theta]$, an irreducible element $z$ of $ S$ can be uniquely written in the form $ z \, =\,  f + Xs$, where $f \in k[\theta], \, s \in  S.$ After normalising by multiplying $z$ by a suitable unit in $ S$ - that is, by a suitable element from $ R \setminus X R$ - there are the following three mutually exclusive possibilities for $z$:
\begin{eqnarray*}
\textbf{(A)} \qquad z &=& X \textit{ or } z = \theta, \qquad \qquad f
= 0, \, s = 1 \textit{ or } f= \theta, \, s= 0; \\
\textbf{(B)} \qquad z &=& 1 + Xs, \qquad \qquad \qquad s \in  S\setminus  R;\\
\textbf{(C)} \qquad z &=& f + Xs, \qquad \qquad f \in k[\theta]\setminus k,\, f \, \textit{monic}, \, s \in  S, \, z \neq \theta.\\
\end{eqnarray*}
\noindent Note that we forbid $s \in  R$ in type $(\mathbf{B})$ in order to exclude units from the list. 

With $R$ denoting either of our possible coefficient algebras $k[[X]]$ or $k[X]_{\langle X \rangle}$, let $Q$ denote the field of fractions of $R$, so $Q = R\langle X^{-1}\rangle$, and observe that $\alpha$ extends to a $k$-algebra automorphism of $Q$. Now  set 
$$T \; := \; Q[\theta;\alpha] \; = \; S\langle X^{-1}\rangle. $$
There is a notion of {\it similarity} (first introduced by Ore) for irreducible elements of the principle left and right ideal domain $T$, which is recalled in detail in \cite[Definition 5.3, Theorem 5.4]{BCM2} and which yields an equivalence relation on the irreducible elements of $S$. These irreducible elements of $S$ then provide a very convenient method to classify the faithful simple and 1-critical cyclic $S$-modules. This is described in detail in \cite[Theorems 5.5 and 5.10]{BCM2}. In brief, every faithful simple (right) $S$-module has the form $S/bS$ for a Type {\bf (B)} irreducible element $b \in S$, and conversely; and given a Type {\bf (B)} irreducible and an element $b'$ of $S$, $S/bS \cong S/b'S$ if and only if $b'$ is a Type {\bf (B)} irreducible of $S$ which is similar to $b$. The result for faithful finitely generated 1-critical $S$-modules is similar, though slightly more complicated to state as the equivalence relation between such modules $A$ and $B$ is necessarily no longer isomorphism, but {\it hull-similarity}, meaning that $A$ is equivalent to $B$ if they have an isomorphic non-zero submodule. In particular, every faithful finitely generated 1-critical right $S$-module is hull-similar to $S/cS$ for a Type {\bf (C)} irreducible element $c$ of $S$, and for every Type {\bf (C)} irreducible $c \in S$, $S/cS$ is a faithful 1-critical $S$-module.

\medskip

Since it lies at the heart of our strategy we restate here the relevant three of the five equivalent statements from \cite[Theorem 6.7]{BCM2}. Naturally, a Type {\bf (B)} $T$-module [{\it resp.} $S$-module] is a module of the form $T/bT$ [{\it resp.} $S/bS$], for $b \in S$ irreducible of Type {\bf (B)}; and analogously for Type {\bf (C)} modules. So Type {\bf (B)} $T$-modules and $S$-modules are simple, as also are Type {\bf (C)} $T$-modules. But Type {\bf (C)} $S$-modules are 1-critical. 

\begin{theorem}\label{target} Retain all the notation introduced so far in $\S$\ref{fail}. The following statements are equivalent.
\begin{enumerate}
\item[(i)] (Right) $(\diamond)$ holds for $ S$.
\item[(ii)] There does not exist a uniserial right $T$-module with composition length 2 whose socle is a type {\bf (B)} simple and whose simple image is type {\bf (C)}.
\item[(iii)] Let $b, c \in  S$ be irreducible, with $b$ type {\bf (B)} and $c$ type {\bf (C)}. Then there exist $b', c'$ in $ S$, respectively irreducibles of types {\bf (B)} and {\bf (C)} and respectively similar to $b$ and to $c$, with
\begin{equation}\label{done} cb \; = \;   b'c'.
\end{equation}
\end{enumerate}
\end{theorem}

\medskip

Henceforth in this section we specialise to the case where $R = k[X]_{\langle X \rangle}$. To prove Theorem \ref{k[X]thm} we'll exhibit irreducible elements $b$ and $c$ of $S$, respectively of Types {\bf (B)} and {\bf (C)}, for which no pair of elements $b', c'$ exists yielding the equation (\ref{done}). To do so some technical preparation is needed. 

\medskip

We shall need the following notions concerning polynomials in $k[X]$.  Let $s\geq 1$ and  $ g=\prod_{i=1}^{s}(X-z_i) \in k[X] $,  where $Z_g:=\{z_1,\ldots, z_s\}$ denotes the set of enumerated roots of $g$. Let $f\in k[X]$ be a monic polynomial and $0\ne \xi\in k$. Suppose that every root of $f$ (counting multiplicity) belongs either to $\xi Z_g$ or $\xi^2Z_g$. Thus there  are (possibly empty) subsets $A$ and $B$ of $Z_g$   and  monic polynomials  $f_A, f_B\in k[X]$ such that $f=f_Af_B$, the set $Z_{f_A}$ of all roots  of  $f_A$ is equal to $\xi A$ and $Z_{f_B}=\xi^2B$. Let $C:=A\cap B$, $A_0:=A\setminus C$,  $B_0:=B\setminus C$ and $D:=Z_g\setminus (A\cup B)$, so that  $Z_g $ is the disjoint union of the sets $A_0,B_0,C,  D$.  The decomposition of $f$ as $f_Af_B$ is said to be a $(g,\xi)$-{\it presentation} of $f$. Notice that although the polynomials $f_A,f_B$ are uniquely determined, the sets $A $ and $B$ are not (consider for example $g:=(X-1)^3$, $f:=(X-\xi)(X-\xi^2)$). Notice that if  $c\in C$ and there exists $z\in D$ such that $c=z$ (i.e. the roots $c$ and $z$  of $g$ are equal), then $\xi z$ is a root of $f_A$ and $\xi^2z$ is a root of $f_B$. Thus replacing $A$ by $A'=( A\setminus  \{c\})\cup \{z\}$ we have $f_A=f_{A'}$ and $|A'\cap B|< |C|$.  We say that the $(g,\xi)$-presentation of $f=f_Af_B$ is $(g,\xi)$-{\it irreducible} if  the cardinality of $A\cap B$ is the smallest possible.  By the above we have:

 \begin{lemma}\label{irrepresentation} Keeping the above notation. Suppose that $f=f_Af_B$ is a $(g,\xi)$-irreducible  presentation of $f$. Then, for any $c\in C$ and $d\in D$, $d\ne c$ (as elements of the base field $k$).
 \end{lemma}

\begin{lemma}\label{q algebraic} Given $S$ as before suppose that
 \begin{equation}\label{E6}
(1+X+\theta+X\theta^2)(1-X+X\theta)=(1+Xt\theta)(a+b\theta+c\theta^2)
\end{equation}
for some  $t, a,b,c \in k[X]_{<X>}$. Then either $k$ has characteristic $2$ or $q$ is a root of unity.
\end{lemma}
\begin{proof}
Comparing the coefficients of the monomials in the equation (\ref{E6}) above we have that
\begin{equation}\label{E71}
1-X^2=a
\end{equation}
\begin{equation}\label{E72}
1-qX+X+X^2=Xt\alpha(a)+b
\end{equation}
\begin{equation}\label{E73}
qX+(1-q^2X)X=Xt\alpha(b)+c
\end{equation}
\begin{equation}\label{E74}
q^2X^2=Xt\alpha(c)
\end{equation}

Substituting $a$ from (\ref{E71}) in (\ref{E72}) we get
\begin{equation}	\label{E8}
1-qX+X+X^2=Xt(1-q^2X^2)+b
\end{equation}
and substituting $b$ from (\ref{E8}) in  (\ref{E73}) we obtain
\begin{equation}\label{E9}
c=qX+(1-q^2X)X-Xt[1-q^2X+qX+q^2X^2-qX\alpha(t)(1-q^4X^2)].
\end{equation}
Note that by (\ref{E73}) $c\in XR$, and by (\ref{E74}) $t$ is a unit in $R$. Write $c=Xc_1$, so that by (\ref{E74})
\begin{equation}\label{E10}
c_1=q\alpha^{-1}(t^{-1}).
\end{equation}
From (\ref{E9}), (\ref{E10}) and applying $\alpha$ we get
\begin{equation}\label{reduced}
\begin{split}
q&=(1+q-q^3X)t\\&-(1+(1-q)q^2X+q^4X^2)t\alpha(t)\\&+q^2X(1-q^6X^2)t\alpha(t)\alpha^2(t).
\end{split}
\end{equation}

\noindent The problem is now reduced to deciding if there exists $t\in R$ that satisfies (\ref{reduced}). We can embed $R=k[X]_{<X>}$ in $k[[X]]$ and write $t=\displaystyle \sum_{i\geq 0} t_iX^i$ for $t_i \in k$. From (\ref{reduced}) it follows that
$$q=(1+q)t_0-t_0^2.$$
Hence
\begin{equation}\label{t0}t_0=1 \; \mbox{or}\; t_0=q.
\end{equation}
\noindent Since $t\in k[X]_{<X>}$, we can write $t=\lambda\frac{f}{g}$ for some $\lambda\in k$, and  $f,g\in k[X]$ monic and co-prime such that $g\notin X k[X]$.  Let $f_0, g_0$ be respectively the constant terms of $f$ and $g$, so that by (\ref{t0})
\begin{equation}\label{degree0}
\lambda f_0= qg_0\quad \mbox{or}\quad \lambda f_0=g_0.
\end{equation}

\noindent Now (\ref{reduced}) can be written as
\begin{equation}\label{*}
\begin{split}
qg\alpha(g)\alpha^2(g)&=\left((1+q)-q^3X\right)\lambda f\alpha(g)\alpha^2(g)\\
&-\left(1+(1-q)q^2X+q^4X^2\right)\lambda^2f\alpha(f)\alpha^2(g)\\
&+q^2X(1-q^6X^2)\lambda^3f\alpha(f)\alpha^2(f).
\end{split}
\end{equation}

\noindent Fix $m := \deg(g)$ and $n := \deg(f)$. If $m\leq n$ then in equation (\ref{*}) the LHS has degree $3m$ while the RHS has degree $3+3n$, a contradiction. If $m>n+1$, the LHS has degree $3m$ while the RHS has degree strictly less than $3m$. So $m=n+1$.

\medskip

Observe that if the conclusion of the lemma does not hold then it will not hold when $k$ is replaced by its algebraic closure. Hence, without of lost of  generality we may assume that $k$ is algebraically closed. Let  $Z_g:=\{z_1,\ldots, z_{n+1}\}$ denote the enumerated roots of $g$.  Since $f$ and $g$ are co-prime, by (\ref{*}) $f$ is a polynomial of degree $n$ such that every root of $f$ is a root either of $\alpha(g) $ or $\alpha^2(g) $. Thus the polynomials $f$ and $g$ satisfy the assumptions of Lemma \ref{irrepresentation} with $\xi=q^{-1}$. Fix the $(g,q^{-1})$-irreducible presentation $f=f_Af_B$ of $f$, so that $A$ and $B$ are  subsets of  $Z_g$ and  $f_A, f_B\in k[X]$ are monic polynomials  with the set $Z_{f_A}$ of all roots of  $f_A$ equal to $q^{-1}A$ and $Z_{f_B}=q^{-2}B$. Moreover the cardinality of  $C:=A\cap B$ is minimal for sets $A,B$ satisfying these conditions. Now we claim that 
\begin{equation}\label{Cempty} C\; =\; \emptyset.
\end{equation}

Write  $f_A=\prod_{z\in A}(X-q^{-1}z)$ and $f_B=\prod_{z\in B}(X-q^{-2}z)$. Note that

\begin{eqnarray}
\alpha(g) &=& q^{n+1}\prod_{i=1}^{n+1}(X-q^{-1}z_i)=q^{n+1}f_A\prod_{z\in B_0}(X-q^{-1}z)\prod_{z\in D}(X-q^{-1}z)\label{2fA},\\
\alpha^2(g) &=& q^{2n+2}f_B\prod_{z\in A_0}(X-q^{-2}z)\prod_{z\in D}(X-q^{-2}z).\label{2fB}
\end{eqnarray}

\noindent Combining (\ref{*}) with (\ref{2fA}) and (\ref{2fB}) we have (after cancelling $f_Af_B\prod_{z\in A_0}(X - q^{-2}z)$)



\begin{equation}\label{2***}
\begin{split}
&q^{3n+4}g \prod_{z\in B_0}(X-q^{-1}z)\prod_{z\in D}(X-q^{-1}z)\prod_{z\in D}(X-q^{-2}z)=\\
&=\left((1+q)-q^3X\right)\lambda q^{3n+3}f_A\prod_{z\in B_0}(X-q^{-1}z)\prod_{z\in D}(X-q^{-1}z)f_B\prod_{z\in D}(X-q^{-2}z)\\
&-\left(1+(1-q)q^2X+q^4X^2\right)\lambda^2q^{3n+2}\prod_{z\in A}(X-q^{-2}z)\prod_{z\in B}(X-q^{-3}z)f_B\prod_{z\in D}(X-q^{-2}z)\\
&-q^{n+8}X(X-\frac{1}{q^3})(X+\frac{1}{q^3})\lambda^3\prod_{z\in C}(X-q^{-2}z)\prod_{z\in B}(X-q^{-3}z)\alpha^2(f)
\end{split}
\end{equation}

Now $\prod_{z\in C}(X-q^{-2}z)$ divides $f_B$ since $C\subseteq B$, so $\prod_{z\in C}(X-q^{-2}z)$ divides the RHS of equation (\ref{2***}). Since for each $z\in C, (X-q^{-2}z)$ divides $f$, it does not divide $g$. Also, by Lemma \ref{irrepresentation},  for each $z\in C, (X-q^{-2}z)$ does not divide $\prod_{z\in D}(X-q^{-2}z)$. It therefore follows that
\begin{equation}\label{division}\prod_{z\in C}(X-q^{-2}z) \textit{ divides } \prod_{z\in B_0}(X-q^{-1}z)\prod_{z\in D}(X-q^{-1}z).\end{equation}

\noindent Hence, given any $z\in C$, either there is a $b\in B_0$ such that $q^{-2}z=q^{-1}b$ or there is a $d\in D$ such that $q^{-2}z=q^{-1}d$. In both these cases $q^{-1}z=b$ or $q^{-1}z=d$ it follows that  $q^{-1}z$  is a root of $f_A$ and of $g$, contradicting the fact that $f$ and $g$ are coprime. So $C=\emptyset$ - that is,  (\ref{Cempty}) is proved.

\medskip

Now we are able to complete the proof of the lemma. Without lost of generality assume that $D=Z_g\setminus (A\cup B)=\{z_{n+1}\}$. By (\ref{Cempty}), we can now rewrite (\ref{2***}) as







\begin{equation}\label{***}
\begin{split}
&q^{3n+4}g \prod_{z\in Z_g\backslash A}(X-q^{-1}z)(X-q^{-2}z_{n+1})=\\
&=\left((1+q)-q^3X\right)\lambda q^{3n+3}f_A\prod_{z\in Z_g\backslash A}(X-q^{-1}z)f_B(X-q^{-2}z_{n+1})\\
&-\left(1+(1-q)q^2X+q^4X^2\right)\lambda^2q^{3n+2}\prod_{z\in A}(X-q^{-2}z)\prod_{z\in B}(X-q^{-3}z)f_B(X-q^{-2}z_{n+1})\\
&-q^{n+8}X(X-\frac{1}{q^3})(X+\frac{1}{q^3})\lambda^3\prod_{z\in B}(X-q^{-3}z)\alpha^2(f).
\end{split}
\end{equation}

\noindent We claim that 
\begin{equation}\label{ellclaim} \exists \, \ell \textit{ such that } 0\leq \ell \leq |B| \textit{ and } z_{n+1}=\pm q^{-\ell -1}.
\end{equation}

Since $f$ and $g$ are coprime, so are $\alpha^2(f)$ and $\alpha^2(g)$.  Hence $(X-q^{-2}z_{n+1})$ does not divide $\alpha^2(f)$. Also $z_{n+1}\neq 0$, as $g\notin Xk[X]$,  and from equation (\ref{***}) we have
\begin{equation}\label{zn}z_{n+1}=\pm \frac{1}{q}\quad \vee\quad \exists b_1\in B: z_{n+1}=q^{-1}b_1. \end{equation}
If the first case happens then (\ref{ellclaim}) follows with $l=0$. Suppose now that the second case of (\ref{zn}) holds. Then equation (\ref{***}) can be written as
\begin{equation}\label{simp_1}
\begin{split}
&q^{3n+4}g \prod_{z\in Z_g\backslash A}(X-q^{-1}z)=\\
&=\left((1+q)-q^3X\right)\lambda q^{3n+3}f_A\prod_{z\in Z_g\backslash A}(X-q^{-1}z)f_B\\
&-\left(1+(1-q)q^2X+q^4X^2\right)\lambda^2q^{3n+2}\prod_{z\in A}(X-q^{-2}z)\prod_{z\in B}(X-q^{-3}z)f_B\\
&-q^{n+8}X(X-\frac{1}{q^3})(X+\frac{1}{q^3})\lambda^3\prod_{z\in B\backslash\{b_1\}}(X-q^{-3}z)\alpha^2(f)
\end{split}
\end{equation}

Let $b'_1,\ldots,b'_{|B|}$ be an enumeration of the roots of $g$ that are in $B$, assuming without loss of generality that 
\begin{equation}\label{add1} b'_1 \; = \; qz_{n+1} \;= \; b_1.
\end{equation}
(Note that we can have multiple roots, hence it may happen that $b'_i=b'_j$ as elements of $k$ for $i\neq j$). If $q^2z_{n+1}$ is one of the elements of the list $b'_2,\ldots, b'_{|B|}$ we let $b_2=q^2z_{n+1}$. Without loss of generality we can assume that $b_2=b'_2$. We proceed with $q^3z_{n+1}$, if it is in the list $b'_3,\ldots,b'_{|B|}$ then we write $b_3=q^3z_{n+1}$, reordering the roots in the list $b'_3,\ldots,b'_{|B|}$ we may assume that $b_3=b'_3$. Iterate the procedure until we reach $l$ such that $b_i=q^{i}z_{n+1}=b'_i$ for all $i\leq l$ but $q^{l+1}z_{n+1}$ is not in the list $b'_{l+1},\ldots, b'_{|B|}$. Note that $l\leq |B|$ and that
$$q^{-1}z_{n+1}=q^{-2}b_1=q^{-3}b_2,$$
$$ q^{-1}b_i=q^{-2}b_{i+1}=q^{-3}b_{i+2}, \forall i\in\{1,\ldots, l-2\},$$
$$q^{-1}b_{l-1}=q^{-2}b_l.$$
Now we have by (\ref{add1}) that
$$\prod_{z\in\{b_1,\ldots,b_{l-2}\}\cup\{z_{n+1}\}}(X-q^{-1}z)=\prod_{z\in\{b_1,\ldots,b_{l-1}\}}(X-q^{-2}z)=\prod_{z\in\{b_2,\ldots,b_{l}\}}(X-q^{-3}z).$$

So, from (\ref{simp_1}) and after cancelling the factor above, we get
\begin{equation}\label{simp_1*}
\begin{split}
&q^{3n+4}g \prod_{z\in B\backslash \{b_1,\ldots,b_{l-2}\}}(X-q^{-1}z)=\\
&=\left((1+q)-q^3X\right)\lambda q^{3n+3}f_A\prod_{z\in Z_g\backslash A}(X-q^{-1}z)\prod_{z\in B\backslash \{b_1,\ldots,b_{l-1}\}}(X-q^{-2}z)\\
&-\left(1+(1-q)q^2X+q^4X^2\right)\lambda^2q^{3n+2}\prod_{z\in A}(X-q^{-2}z)\prod_{z\in B}(X-q^{-3}z)\prod_{z\in B\backslash \{b_1,\ldots,b_{l-1}\}}(X-q^{-2}z)\\
&-q^{n+8}X(X-\frac{1}{q^3})(X+\frac{1}{q^3})\lambda^3\prod_{z\in B\backslash\{b_1,\ldots,b_l\}}(X-q^{-3}z)\alpha^2(f)
\end{split}
\end{equation}

Now note that $(X-q^{-1}b_{l-1})$ divides the LHS and is equal to $(X-q^{-2}b_l)$, which divides $\prod_{z\in B\backslash \{b_1,\ldots,b_{l-1}\}}(X-q^{-2}z)$. Since $q^{-2}b_l$ is a root of $\alpha^2(g)$ can't be a root of $\alpha^2(f)$, neither can it be $0$. Hence either $q^{-2}b_l=q^{-3}b$ for some $b\in B\backslash\{b_1,\ldots,b_l\}$ or $q^{-2}b_l=\pm q^{-3}$. Notice that the first equation cannot hold as otherwise 
$$q^{l+1}z_{n+1}=qb_l=b\in B\backslash\{b_1,\ldots,b_l\},$$ 
which is impossible because of the definition of $l$. Thus $q^{-2}b_l=\pm q^{-3}$  and consequently $z_{n+1}=q^{-l}b_l=\pm q^{-l-1}$, proving (\ref{ellclaim}).


By equation (\ref{degree0}), we have

\begin{equation}
\lambda q^{-|A|-2|B|}\prod_{i=1}^nz_i=-q\prod_{i=1}^{n+1}z_i\quad \mbox{or}\quad \lambda q^{-|A|-2|B|}\prod_{i=1}^nz_i=-\prod_{i=1}^{n+1}z_i
\end{equation}
hence
\begin{equation}\label{a}
\lambda =q^{|A|+2|B|+1}z_{n+1}=\pm q^{|A|+2|B|-l}=\pm q^{n+|B|-l}\quad \mbox{or}\quad \lambda=\pm q^{n+|B|-l-1}
\end{equation}
for some $0\leq l\leq |B|$.

Looking at the leading coefficient of equation (\ref{*}), we have
$$q^{3n+4}=-q^{3n+6}\lambda-q^{3n+6}\lambda^2-q^{3n+8}\lambda^3$$
so $$0=1+q^2\lambda+q^2\lambda^2+q^4\lambda^3=(\lambda+\frac{1}{q^2})(q^4\lambda^2+q^2)$$
from what follows that either $\lambda=-q^{-2}$ or $\lambda^2=-q^{-2}$
and by (\ref{a}) we have
\begin{equation}
1=\pm q^{n+|B|-l+2}\quad \mbox{or}\quad 1=\pm q^{n+|B|-l+1}
\end{equation}
or
\begin{equation}\label{final}
-1=q^{2n+2|B|-2l+2}\quad \mbox{or}\quad -1=q^{2n+2|B|-2l}
\end{equation}
By the choice of $l$, $n+|B|-l+1\geq n+1$ and if $\mathrm{char}(k)\neq 2$ $q$ is a root of unity.


\end{proof}

We remark inpassin that in the case $\mathrm{char}(k)=2$, the second  equality of (\ref{final}) holds if $n=0=|B|$. Then from equation (\ref{zn}) $z_{n+1}=\frac{1}{q}$. An easy computation shows that $\lambda =\frac{1}{q}$, $f=1$ and $g=(X+\frac{1}{q})$ satisfies (\ref{*}), i.e. the assumption of Lemma \ref{q algebraic} holds for any $0\ne q\in k$. In fact
$(1+X+\theta+X\theta^2)(1-X+X\theta)=[1+X(1+qX)^{-1}\theta][(1+X^2)+(1+qX+(1+q)X^2)\theta+ qX(X+1)\theta^2]$.

\begin{lemma}\label{L2} Given $S$ as before suppose that the  polynomial  $h=(1+X+\theta+X\theta^2)(1-X+X\theta)$  is   divisible on the right by a polynomial $u+t\theta$ with $u,t\in R$. If $t$ is invertible in $R=k[X]_{<X>}$ then either $k$ has characteristic $2$ or $q$ is a root of unity.
\end{lemma}
\begin{proof}
 Suppose that $h$ is divisible on the right by $d=u+t\theta$. Since the term in $h$ of degree $0$ in $\theta$ is invertible in $R$, $u$ has to be invertible as well. Thus, replacing $d$ by $u^{-1}d$, we may assume that $d=1+t\theta$. Assume now that $t$ is invertible and there are elements  $a,b,c\in R$ such that
\begin{equation}\label{E1}
(1+X+\theta+X\theta^2)(1-X+X\theta)=(a+b\theta+c\theta^2)(1+t\theta).
\end{equation}
Comparing  coefficients in (\ref{E1}) we have
\begin{eqnarray}
1-X^2&=&a\\
1-qX+X+X^2&=&b+at\\
X(1-q^2X)+qX&=&c+b\alpha(t)\\
q^2X^2&=&c\alpha^2(t)
\end{eqnarray}
Hence
\begin{eqnarray}
a&=&1-X^2\\
b&=&1-qX+X+X^2-(1-X^2)t\\
X(1+q-q^2X)&=&c+[1-qX+X+X^2-(1-X^2)t]\alpha(t)\label{16}\\
c&=&q^2X^2\alpha^2(t^{-1})
\end{eqnarray}
and from (\ref{16}) we have
\begin{equation}\label{E2}
X(1+q-q^2X)t^{-1}\alpha(t^{-1})=q^2X^2t^{-1}\alpha(t^{-1})\alpha^2(t^{-1})+(1-qX+X+X^2)t^{-1}-(1-X^2)
\end{equation}
or
\begin{equation}\label{E2b}
q^2X^2t^{-1}\alpha(t^{-1})\alpha^2(t^{-1})= (1-X^2)+ X(1+q-q^2X)t^{-1}\alpha(t^{-1})-(1-qX+X+X^2)t^{-1}
\end{equation}
Working modulo $<X>$, it follows that $1-t^{-1}\in <X>$. Write $t^{-1}=\lambda\frac{f}{g}$ for some $\lambda \in k$ and $f,g\in k[X]$ monic and co-prime such that $f,g\notin <X>$. If $f_0, g_0$ are the constant terms of the polynomials $f$ and $g$ respectively, we have
\begin{equation}\label{zero_degree}
g_0=\lambda f_0.
\end{equation}
Multiplying equation (\ref{E2b}) by $g\alpha(g)\alpha^2(g)$ we obtain the following equation in $k[X]$:
\begin{equation}\label{E2c}
\begin{split}
q^2X^2\lambda^3f\alpha(f)\alpha^2(f)&= (1-X^2)g\alpha(g)\alpha^2(g)\\
&+ X(1+q-q^2X)\lambda^2f\alpha(f)\alpha^2(g)\\
&-(1-qX+X+X^2)\lambda f\alpha(g)\alpha^2(g)
\end{split}
\end{equation}

Making use of equation (\ref{E2c}) one can easily check that $f$ and $g$ have the same degree, say $n$, and that any root of $\alpha^2(g)$ is either a root of $f$ or of $\alpha(f)$. So any root of $g$ is either a root of $\alpha^{-2}(f)$ or of $\alpha^{-1}(f)$.  Looking at the leading coefficient of equation (\ref{E2c}) we have
\begin{equation}
\lambda^3q^{3n+2}=-q^{3n}-\lambda^2q^{3n+2}-\lambda q^{3n}.
\end{equation}
Hence
\begin{equation}\label{lb}
\lambda=-1\qquad\mbox{or}\qquad\lambda^2=-q^{-2}
\end{equation}

As in Lemma \ref{q algebraic}, if necessary we can replace the base field $k$ by its algebraic closure. Let $Z_f=\{z_1,\ldots, z_{n}\}$ denote  the set of enumerated roots of $f$ and  fix the $(f,q)$-irreducible presentation $g=g_Ag_B$ of $g$. Thus $A,B$ are  subsets of  $Z_f$ and   $g_A, g_B\in k[X]$ are monic polynomials   with the set $Z_{g_A}$ of all roots of  $g_A$ is equal to $qA$ and $Z_{g_B}=q^{2}B$ and cardinality of  $C=A\cap B$ minimal among sets $A,B$ as above. Let $A$ be the disjoint union of $A_0$ and of $C$, $B$ the disjoint union of $B_0$ and of $C$ and $D=Z_f\backslash (A\cup B)$.

Write
\begin{equation}\label{fAfB}
\begin{split}
g_A=\prod_{z\in A_0}(X-qz)\prod_{z\in C}(X-qz),\\
g_B=\prod_{z\in B_0}(X-q^{2}z)\prod_{z\in C}(X-q^{2}z).
\end{split}
\end{equation}
\noindent If $C=\emptyset$, since $g=g_Ag_B$ by (\ref{zero_degree}) we get
\begin{equation}\label{lb2}
q^{|A|+2|B|}\prod_{z\in {Z_f}} z=\lambda\prod_{z\in Z_f} z
\end{equation}
taking into account (\ref{lb}) and (\ref{lb2}) it follows that either $-1=q^{|A|+2|B|}$ or $-q^{-2}=q^{2|A|+4|B|}$, so either $q$ is a root of unity or $|A|=|B|=0$ and $k$ has characteristic $2$, as required.

To complete the proof we claim that $C=\emptyset$. Assume otherwise. Since
$$\alpha^2(g)=q^{2n}\prod_{z\in A}(X-q^{-1}z)\prod_{z\in B}(X-z)$$
and
$$f\alpha(f)\alpha^2(f)=q^{3n} \prod_{z\in Z_f}(X-z)\prod_{z\in Z_f}(X-q^{-1}z)\prod_{z\in Z_f}(X-q^{-2}z),$$ dividing equation (\ref{E2c}) by $\alpha^2(g)$ we get

\begin{equation}\label{niewieder1}
\begin{split}
q^2X^2\lambda^3 q^n\prod_{z\in Z_f\backslash B}(X-z) \prod_{z\in Z_f\backslash A}(X-q^{-1}z) &\prod_{z\in Z_f}(X-q^{-2}z)= \\ &=(1-X^2)g\alpha(g)\\
&+ X(1+q-q^2X)\lambda^2f\alpha(f)\\
&-(1-qX+X+X^2)\lambda f\alpha(g)
\end{split}
\end{equation}

Since $\prod_{z\in A_0}(X-z)$ divides $f$ and $\alpha(g)=q^n\prod_{z\in A}(X-z)\prod_{z\in B}(X-qz)$ and $Z_f\backslash B=D\cup A_0$, we have

\begin{equation}\label{niewieder2}
\begin{split}
q^2X^2\lambda^3 q^n&\prod_{z\in D}(X-z) \prod_{z\in Z_f\backslash A}(X-q^{-1}z) \prod_{z\in Z_f}(X-q^{-2}z)= \\ &=(1-X^2)q^n\prod_{z\in A}(X-qz)\prod_{z\in B}(X-q^2z)\prod_{z\in C}(X-z)\prod_{z\in B}(X-qz)\\
&+ X(1+q-q^2X)\lambda^2\prod_{z\in Z_f\backslash A_0}(X-z)q^n\prod_{z\in Z_f}(X-q^{-1}z)\\
&-(1-qX+X+X^2)\lambda \prod_{z\in Z_f\backslash A_0}(X-z)q^n\prod_{z\in A}(X-z)\prod_{z\in B}(X-qz)
\end{split}
\end{equation}

By considering the RHS of (\ref{niewieder2}), given $c\in C$, 
 \begin{equation}\label{add2} c \textit{ is a root of }X^2\lambda^3 q^{n}\prod_{z\in D}(X-z) \prod_{z\in Z_f\backslash A}(X-q^{-1}z) \prod_{z\in Z_f}(X-q^{-2}z).
 \end{equation}
 But $c \in C = A \cap B$, and $g = \Pi_{z \in A}(X - qz) \Pi_{z \in B}(X - q^2 z).$ Hence $c$ is a root of $\alpha(g)$ and of $\alpha^2(g)$. However, since $f$ and $g$ are coprime,  $c$ is neither a root of $\alpha(f)$ nor of $\alpha^2(f)$. Hence $c\notin q^{-1}Z_f$ and $c\notin q^{-2}Z_f$. So, by (\ref{add2}), either $c=0$ or $c\in D$. But $g\notin <X>$ and by Lemma \ref{irrepresentation} with $\xi=q$, we conclude that $C=\emptyset$.
\end{proof}

Note that, in the setting of the above lemma, if the characteristic of $k$ is $2$ then $q$ is not necessarily a root of unity. In fact
\begin{equation}
(1+X+\theta+X\theta^2)(1-X+X\theta)=((1+X^2)+(1+q)X\theta+q^2X^2\theta^2)(1+\theta)
\end{equation}
for any $q\in k^*$.

\bigskip

\noindent {\it Proof of Theorem \ref{k[X]thm}}
 It is known (Cf. \cite[Example 6.9]{BCM2}) that the  polynomial $v=1+X+\theta+X\theta^2\in S$ is an irreducible element of  type {\bf{(C)}} while $1-X+X\theta$ is an irreducible element of type {\bf{(B)}} in $S$.

Suppose   $q\in k$ is not a root of unity and that $char(k)\neq 2$. Then Lemmas \ref{L2} and \ref{q algebraic} guarantee that the product $vw$ cannot be written as product of an irreducible elements of type {\bf{(B)}} and a one of type {\bf{(C)}}. This shows that $S$ does not satisfy the monoid commutativity  condition of \cite[Threorem 1.3(2)]{BCM2}. Thus, by that theorem, $S$   doesn't have $(\diamond)$ for right $S$-modules.

\bigskip

Since  $q\in k$ is a root of unity if and only if $q^{-1}$ is, applying the above theorem to the opposite ring of $S$, we get
\begin{corollary}
 If $q\in k$ is not a root of unity and $char(k)\neq 2$, 
 then $S$ does not satisfy  $(\diamond)$ for left $S$-modules.
\end{corollary}
It is known that if $q$ is a root of unity, then $S$ satisfies $(\diamond)$ property. Thus, Theorem  \ref{k[X]thm} yields also:
\begin{corollary}
Suppose $\mathrm{char}(k)\neq 2$. Then  $S$   satisfies $(\diamond)$   if and only if $q$ is a root of unity.
\end{corollary}

\bigskip

\section*{Acknowledgements}{ }

Part of this work was done when the first and third authors visited the University of Porto in October 2023. They thank the university and its staff for their hospitality. The third author thanks the support of CMUP. The work of the first author was supported by Leverhulme Emeritus Fellowship EM-2017-081/9.

Some of this work was done while the second author visited the University of Warsaw in March 2023 and in February and March 2024. She would like to thank the University of Warsaw for their support. The second author was also partially supported by CMUP, member of LASI, which is financed by national funds through FCT – Funda\c{c}\~ao para a Ci\^encia e a Tecnologia, I.P., under the projects with reference UIDB/00144/2020 and UIDP/00144/2020.

We are thankful to Joachim Jelisiejew for providing us a short argument for the proof of the first part of Proposition \ref{Jelisi}.

\bigskip

  \textsc{School of Mathematics and Statistics, University of Glasgow, Glasgow G12 8QQ, Scotland}\par\nopagebreak
  \textit{E-mail address}: \texttt{Ken.Brown@glasgow.ac.uk}

  \medskip
   \textsc{CMUP, Departamento de Matem\'atica, Faculdade de Ci\^encias, Universidade do Porto,  Rua do Campo Alegre s/n, 4169-007 Porto, Portugal}\par\nopagebreak
  \textit{E-mail address}: \texttt{pbcarval@fc.up.pt}

  \medskip

   \textsc{Institute of Mathematics, University of Warsaw,  Banacha 2, 02-097 Warsaw, Poland}
  \par\nopagebreak
  \textit{E-mail address}:   \texttt{jmatczuk@mimuw.edu.pl}

\medskip

 \end{document}